\documentclass[12pt]{amsart}

\usepackage[english]{babel}
\usepackage[utf8x]{inputenc}
\usepackage[T1]{fontenc}
\usepackage{amsmath}
\usepackage{amstext}
\usepackage{amsfonts}
\usepackage{amssymb}
\usepackage{amsthm}
\usepackage{amsrefs}
\usepackage{mathtools}
\usepackage{dsfont}
\usepackage{color}
\usepackage{graphicx}
\usepackage[colorinlistoftodos]{todonotes}
\usepackage[colorlinks=true, allcolors=blue]{hyperref}
\usepackage{float}
\usepackage[colorinlistoftodos]{todonotes}
\usepackage{theoremref}
\allowdisplaybreaks
\newtheorem{theorem}{Theorem}[section]

\newtheorem{prop}[theorem]{Proposition}
\newtheorem{cor}[theorem]{Corollary}

\newtheorem{lem}[theorem]{Lemma}

\newtheorem*{cor*}{Corollary}
\newtheorem*{thm*}{Theorem}
\newtheorem*{lem*}{Lemma}

\newtheorem*{prop*}{Proposition}
\theoremstyle{definition}

\newtheorem{defn}[theorem]{Definition}
\newtheorem{example}[theorem]{Example}
\newtheorem*{defn*}{Definition}

\theoremstyle{remark}
\newtheorem{remark}[theorem]{Remark}

\newcommand{\pr}{\operatorname{Prob}}
\newcommand{\prg}{\pr(\Gamma)}

\newcommand{\Rad}{\operatorname{Rad}}

\usepackage[nameinlink]{cleveref} 

\DeclareRobustCommand{\SkipTocEntry}[5]{}

\crefname{chapter}{chapter}{chapters}
\Crefname{chapter}{Chapter}{Chapters}
\crefname{section}{section}{sections}
\Crefname{section}{Section}{Sections}
\crefname{subsection}{section}{sections}
\Crefname{subsection}{Section}{Sections}
\crefname{subsubsection}{section}{sections}
\Crefname{subsubsection}{Section}{Sections}
\crefname{figure}{figure}{figures}
\Crefname{figure}{Figure}{Figures}
\crefname{table}{table}{tables}
\Crefname{table}{Table}{Tables}

\crefname{theorem}{theorem}{theorems}
\Crefname{theorem}{Theorem}{Theorems}
\crefname{proposition}{proposition}{propositions}
\Crefname{proposition}{Proposition}{Propositions}
\crefname{corollary}{corollary}{corollaries}
\Crefname{corollary}{Corollary}{Corollaries}
\crefname{lemma}{lemma}{lemmas}
\Crefname{lemma}{Lemma}{Lemmas}
\crefname{definition}{definition}{definitions}
\Crefname{definition}{Definition}{Definitions}
\crefname{conjecture}{conjecture}{conjectures}
\Crefname{conjecture}{Conjecture}{Conjectures}
\crefname{example}{example}{examples}
\Crefname{example}{Example}{Examples}
\crefname{remark}{remark}{remarks}
\Crefname{remark}{Remark}{Remarks}

\title[Faithful Actions]{Faithful actions on Generalized Furstenberg boundary}
\author{Farid Behrouzi}
\address{Farid Behrouzi\newline\hspace*{1em} Department of Pure Mathematics, Faculty of Mathematical Sciences, Alzahra University, Tehran, Iran}
\email{f\_behrouzi@alzahra.ac.ir}
\author{Zahra Naghavi}
\address{Zahra Naghavi\newline\hspace*{1em} Department of Pure Mathematics, Faculty of Mathematical Sciences, Alzahra University, Tehran, Iran\newline\hspace*{1em}School of Mathematics, Institute for Research in Fundamental Sciences (IPM), P.O. Box: 19395-5746, Tehran, Iran}
\email{z.naghavi@alzahra.ac.ir, naghavi.zahra@gmail.com}
 \keywords{countable discrete group, $C^*$-algebra, crossed product, minimal space, unique trace, Powers' averaging property, Furstenberg boundary.\newline\hspace*{1em} Second auther partly supported by grant from IPM (No. 1400460031)}

\begin{document}
\begin{abstract}
Let $G$ be a countable discrete group that act minimally on a compact Hausdorff space $X$ by homeomorphisms. Our goal is to establish the equivalence between the faithfulness of the action of $G$ on the generalized Furstenberg boundary $\partial_F(G, X)$ and a weakened version of the generalized Powers' averaging property. This result provides valuable insights into the state space of the crossed product $C(X)\rtimes_{r}G$.
\end{abstract}
\maketitle
\section{Introduction}

In 1958, Kaplansky \cite{Kap} posed the question of whether a simple projectionless $C^*$-algebra exists. Interestingly, this question had been discussed in a conversation with Kadison about ten years earlier. Following Kadison's suggestion, Powers proved in 1968 that the reduced $C^*$-algebra of the free group on two generators, denoted as $C^{*}_{r}(\mathbb{F}_{2})$, is simple. However, Powers published this result seven years later \cite{Po}. It wasn't until 1982 that Pimsner and Voiculescu demonstrated that $C^{*}_{r}(\mathbb{F}_{2})$ is projectionless \cite{PV}. Additionally, Powers' technique showed that $C^{*}_{r}(\mathbb{F}_{2})$ possesses a unique trace. For more historical information, please refer to \cite{V} and \cite{H}.

A group $G$ is said to be $C^*$-simple if its reduced group $C^*$-algebra $C^{*}_{r}(G)$ is a simple algebra. On the other hand, the unique trace property refers to the condition where $C^{*}_{r}(G)$ possesses a unique trace. Following Powers' influential paper, many researchers attempted to characterize $C^*$-simple groups and determine whether $C^*$-simplicity is equivalent to the unique trace property. Finally, after nearly 40 years, Kalantar and Kennedy provided a comprehensive characterization of all discrete $C^*$-simple groups in their important paper \cite{KK}. Their work established the equivalence between $C^*$-simplicity and the freeness of the action $G\curvearrowright\partial_{F}G$, where $\partial_{F}G$ represents the universal Furstenberg boundary. Subsequently, in \cite{BKKO}, it was demonstrated that $G$ possesses the unique trace property if and only if the action $G\curvearrowright\partial_{F}G$ is faithful. These findings indicated that having the unique trace property is a weaker condition compared to $C^*$-simplicity. The crux of the proof lies in the fact that $\ker(G\curvearrowright\partial_{F}G)=\Rad(G)$ \cite{Fur}, where $\Rad(G)$ denotes the largest normal amenable subgroup of $G$, known as the amenable radical. Notably, the authors of both papers did not adopt Powers' method of proof.

Now, let us delve into Powers' technique, known as Powers' averaging property. This property can be described in terms of the set of all finitely supported probability measures denoted by $\pr_{f}(G)$. It should be noted that $\pr_{f}(G)$ is a subsemigroup of the set of all probability measures $\pr(G)$. Let $\tau_{\lambda}$ denote the canonical trace on $C^{*}_{r}(G)$. We say that a group $G$ has the Powers' averaging property if, for any $a\in C^{*}_{r}(G)$ with $\tau_{\lambda}(a)=0$, the following holds:
$$0\in\overline{\{\mu a : \mu\in \pr_{f}(G)\}}.$$
To understand this statement, consider that if $G$ acts on any $C^*$-algebra $A$, there is always an action of $\pr(G)$ on $A$ given by
$$\mu a=\sum_{t\in G}\mu(t)(t.a), \quad (\mu\in\pr(G),\ a\in A).$$

Kennedy \cite{Ken} and Haagerup \cite{Ha} independently established that $C^*$-simplicity is equivalent to the Powers' averaging property. Furthermore, Haagerup showed that $G$ possesses the unique trace property precisely when for all $t\in G\setminus\{e\}$
$$0\in\overline{\{\mu\lambda_{t} : \mu\in \pr_{f}(G)\}},$$
where $\lambda : G\to B(\ell^{2}(G))$ denotes the left regular representation. Notably, this condition is weaker than the Powers' averaging property.

In the context of a discrete group $G$ acting on a compact Hausdorff space $X$ by homeomorphisms, Amrutam and Ursu \cite{AU} have recently generalized the concept of Powers' averaging property by introducing the notion of "$C(X)$-valued probability measures." They define the generalized Powers' averaging property for $C(X)\rtimes_{r}G$ as follows: for every $a\in C(X)\rtimes_{r}G$ with $\mathbb{E}(a)=0$, where $\mathbb{E} : C(X)\rtimes_{r}G\to C(X)$ represents the canonical conditional expectation,
$$0\in\overline{\{\mu a : \mu\in \pr_{f}(G, C(X))\}},$$
where $\pr_{f}(G, C(X))$ denotes the set of finitely supported $C(X)$-valued probability measures.

Furthermore, they establish that for a minimal $G$-space $X$, the $C^*$-algebra $C(X)\rtimes_{r}G$ is simple if and only if it satisfies the generalized Powers' averaging property.

Although the introduction of $C(X)$-valued probability measures in the paper by Amrutam and Ursu is not wrong, it is confusing and not easy to work. In our paper, we provide an accurate definition of $C(X)$-valued probability measures. We also demonstrate that $\ker(G\curvearrowright\partial_{F}(G, X))=\{e\}$ precisely when for all $t\in G\setminus\{e\}$
$$0\in\overline{\{\mu\lambda_{t} : \mu\in \pr_{f}(G, C(X))}\},$$
which extends Haagerup's result on the unique trace property of $G$.

In addition to this introduction, the paper consists of three other sections. In Section 2, we provide the necessary background. In Section 3, we define $C(X)$-valued measures and establish their precise form. In Section 4, we prove that $\ker(G\curvearrowright\partial_{F}(G, X))=\Rad(G)\cap\ker(G\curvearrowright X)$. Then we establish the connection between the faithfulness of the action $G\curvearrowright\partial_{F}(G, X)$ and a weakened version of the generalized Powers' averaging property.

\subsection*{Acknowledgements}
The authors are grateful to Tattwamasi Amrutam for many helpful discussions and suggestions during the completion of this paper. The authors would also like to thank the anonymous referee for the detailed reading of our paper and for their comments and suggestions which enhanced the exposition of the paper.

\section{Preliminaries}
In this paper, we specifically examine a countable discrete group $G$ and its actions on compact Hausdorff spaces. A compact Hausdorff space $X$ is referred to as a $G$-space if it admits a group homomorphism from $G$ into the group of homeomorphisms of $X$. This group action is denoted by $G\curvearrowright X$, indicating that $G$ acts on $X$ by homeomorphisms.

Similarly, a group $G$ acts on a $C^*$-algebra $A$ by $*$-automorphisms if there exists a group homomorphism from $G$ into the group of $*$-automorphisms of $A$. For instance, the action $G\curvearrowright X$ induces an action of $G$ on the commutative $C^*$-algebra $C(X)$ given by
$$(t\cdot f)(x) = f(t^{-1}x), \quad (t\in G,\ f\in C(X),\ x\in X).$$

A linear functional $\phi$ on a $C^*$-algebra $A$ is considered a state if it is unital and positive, meaning that $\phi(1) = 1$ and $\phi(a) \geq 0$ for all $a\in A$ with $a\geq 0$. The set of all states of $A$ is denoted by $S(A)$. A state $\phi$ of $A$ is called a tracial  state if, for all $a, b\in A$, $\phi(ab) = \phi(ba)$.

If $G$ acts on a $C^*$-algebra $A$, then the state space $S(A)$ becomes a $G$-space under the action defined as
$$(t\cdot \phi)(a) = \phi(t^{-1}a), \quad (t\in G,\ a\in A,\ \phi\in S(A)).$$
For example, the action $G\curvearrowright X$ induces an action $G\curvearrowright \pr(X)$, where $\pr(X)$ is the space of all probability measures on $X$, since $\pr(X)$ can be identified with the state space $S(C(X))$.

The term minimal applies to a $G$-action on a space $X$ when the $G$-orbit $Gx$ is dense in $X$ for every $x\in X$.
On the other hand, a $G$-action on $X$ is considered strongly proximal if, for any probability measure $\mu\in\pr(X)$, the weak*-closure of the orbit $G\mu$ contains a Dirac measure $\delta_{x}$ for some $x\in X$. A boundary of a $G$-space $X$ refers to a space that satisfies both the minimal and strongly proximal properties. In other words, a boundary is a $G$-space where every point can be arbitrarily approximated by elements of its $G$-orbit, and for any probability measure on $X$, there exists a point that can be approximated by averaging over the $G$-orbit of that measure.

Furstenberg's work \cite{F} establishes the existence of a unique universal boundary for every group $G$, denoted as $\partial_{F}G$ and referred to as the Furstenberg boundary of $G$. This means that any boundary $X$ can be continuously and equivariantly mapped onto $\partial_{F}G$. Furthermore, in the article \cite{KK}, it is proven that the Furstenberg boundary $\partial_{F}G$ can be identified with the $G$-injective envelope of the complex numbers $\mathbb{C}$, denoted as $\mathrm{I}_{G}(\mathbb{C})$.


Consider a $G$-equivariant map $\varphi: Y\to X$ between $G$-spaces $X$ and $Y$. If $\varphi$ is surjective, we call the pair $(Y, \varphi)$ an extension of $X$. An extension $(Y, \varphi)$ of $X$ is said to be minimal if $Y$ is a minimal space. Furthermore, it is termed a strongly proximal extension if every probability measure $\mu\in\pr(Y)$ with support contained in $\varphi^{-1}(x)$ for some $x\in X$ is strongly proximal. We define a $(G, X)$-boundary as a minimal strongly proximal extension of $X$. Notably, the spectrum of the $G$-injective envelope of $C(X)$, denoted as $\mathrm{I}_{G}(C(X))$, serves as a $(G, X)$-boundary. We denote this unique $G$-space (up to homeomorphism) by $\partial_{F}(G, X)$. For further details, see \cite{Naghavi}.

Now, let $\lambda$ denote the left regular representation of $G$ on the Hilbert space $\ell^{2}(G)$. The reduced $C^*$-algebra $C^*_r(G)$ is defined as follows:
$$C^*_r(G)=\overline{\mathrm{span}\{\lambda_{t}: t\in G\}}^{\vert\vert~\cdot~\vert\vert}.$$

For a $C^*$-algebra $A$, the reduced crossed product of an action $G\curvearrowright A$, denoted by $A\rtimes_{r}G$, is defined as the norm closure of the image of the regular representation $C_{c}(G, A)\to B(H\otimes\ell^{2}(G))$, where $C_{c}(G, A)$ is the linear space of finitely supported functions on $G$ with values in $A$. An element in $C_{c}(G, A)$ is typically represented as a finite sum $x=\sum_{s\in G}a_{s}\lambda_{s}$, where $a_s\in A$ and $\lambda_s$ denotes the left regular representation of $G$. It is worth noting that when $A=\mathbb{C}$, the reduced crossed product $\mathbb{C}\rtimes_{r}G$ can be identified with the reduced $C^*$-algebra $C^{*}_{r}(G)$. The action of $G$ on $A\rtimes_{r}G$ is inner, meaning that for $t\in G$ and $a\in A\rtimes_{r}G$, the action is given by $ta=\lambda_{t}a\lambda_{t}^{-1}$.

Consider an inclusion of $C^*$-algebras $B\subseteq A$. A conditional expectation from $A$ onto $B$ is a completely positive contractive projection $\mathbb{E}: A\to B$ that satisfies $\mathbb{E}(bxb')=b\mathbb{E}(x)b'$ for every $x\in A$ and $b, b'\in B$. In the case of the reduced crossed product, there is always a conditional expectation $\mathbb{E}: A\rtimes_{r}G\to A$ defined as $\mathbb{E}\left(\sum a_{s}\lambda_{s}\right)=a_{e}$.  For more information on crossed products, completely positive maps, and conditional expectations, we recommend referring to \cite{BroOza08}.

\section{The space of $C(X)$-valued probability measures}
A probability measure $\mu\in\prg$ is a positive measure with $\mu(G)=1$. It is worth noting that any $\mu\in\prg$ can be expressed as a sum $\sum_{s\in G}\varepsilon_{s}\delta_{s}$, converges in weak* topology, where $\varepsilon_{s}\geq 0$, $\sum_{s\in G}\varepsilon_{s}=1$, and $\delta_{s}$ denotes the Dirac measure on $s$. This representation allows for a concise description of probability measures on $G$. The objective of this section is to generalize this notion and extend the concept to further settings or structures.
\begin{defn}\label{def11}
Suppose $G$ is a countable discrete group, and $X$ is a compact Hausdorff space. A map $\mu : G\to C(X)$ is called a $C(X)$-valued probability measure on $G$ if it satisfies the following properties:
\begin{itemize}
\item[(i)] $\mu$ is positive, meaning that for every $t\in G$, $\mu(t)\geq 0$.
\item[(ii)] The series $\sum_{t\in G}\mu(t)$ uniformly converges to the constant function $1_{C(X)}$, where $1_{C(X)}$ is the function that takes the value $1$ at every point in $X$.
\end{itemize}
\end{defn}

We use the notation $\pr(G, C(X))$ to refer to the set of all $C(X)$-valued probability measures on $G$, and $\pr_f(G, C(X))$ represents the set of all finitely supported $C(X)$-valued probability measures on $G$.

\begin{remark}
	The $C(X)$-valued probability measures on $G$, as defined in \ref{def11}, can be seen as positive elements of the Banach algebra $\ell^1(G, C(X))$ with norm 1. It is easy to observe that $\pr(G, C(X))$ forms a semigroup under convolution multiplication.
\end{remark}
\begin{remark}
	As mentioned in the introduction, Amrutam and Ursu introduced $C(X)$-valued  probability measures in \cite{AU} as formal sum:
	
	\begin{equation}\label{eq11}
	\sum_{s\in G}\sum_{s\in I_s} f_i sf_i
	\end{equation}
	
	where $I_s$'s are pairwise disjoint sets, satisfying the property that $f_i \geq 0$ and $\sum_{s\in G}\sum_{s\in I_s} f_i ^2=1$. The relationship between this definition and our definition can be expressed through a surjection map that maps a formal sum $\sum_{s\in G}\sum_{s\in I_s} f_i sf_i$ to a function $\mu:G\to C(X)$, where $\mu(s)=\sum_{i\in I_s} f_i^2$. Essentially, a summation like (\ref{eq11}) can be regarded as a function assigning each element $s\in G$ a subset of positive elements $C(X)$, denoted by $L_s$, such that
	
	\[
	\sum_{s\in G}\sum_{f_i\in L_s} f_i^2 = 1.
	\]
\end{remark}

\begin{example}
Suppose $\xi\in\pr(G)$. We define $\mu_{\xi}: G\to C(X)$ as $\mu_{\xi}(t) = \xi(t) \cdot 1_{C(X)}$. It is clear that $\mu_{\xi}\geq 0$ since $\xi(t)\geq 0$ for all $t\in G$, and the function $1_{C(X)}$ is non-negative. Moreover, for any $x\in X$, we have:
$$
\sum_{t\in G}\mu_{\xi}(t)(x) = \sum_{t\in G}\xi(t)\cdot 1_{C(X)}(x) = \sum_{t\in G}\xi(t) = 1,
$$
where the last equality holds because $\xi$ is a probability measure on $G$.

This shows that under the map $\xi\mapsto\mu_{\xi}$, the set $\pr(G)$ can be naturally embedded as a subset of $\pr(G, C(X))$.  In the special case where there exists $s\in G$ such that $\xi=\delta_{s}$, the Dirac delta measure concentrated at $s$, we denote $\mu_{\xi}$ as $\Delta_{s}$. In this case, $\Delta_{s}(t) = \delta_{s}(t)\cdot 1_{C(X)} $ for all $t\in G$.
\end{example}
\begin{lem}
A $C(X)$-valued probability measure $\mu$ on $G$ belongs to $\pr(G, C(X))$ if and only if it can be expressed in the form $\mu = \sum_{s\in G} f_s \Delta_s$, where $0\leq f_s\in C(X)$ for each $s\in G$, and the series $\sum_{s\in G} f_s$ converges uniformly to the constant function $1_{C(X)}$.
\end{lem}

\begin{proof}
The backward direction is obvious. For the forward direction, let $f_{s}=\mu(s)$.
\end{proof}

Given $f, g\in C(X)$, we define $f\mu: G\to C(X)$ by $(f\mu)(t)=f\cdot\mu(t)$, and $\mu g: G\to C(X)$ by $(\mu g)(t)=\mu(t)\cdot g$. It follows that $f\mu g=fg\mu$. This shows that $\pr(G, C(X))$ is $C(X)$-convex, meaning that for any finitely many $f_{1},\ldots, f_{n}\in C(X)$ with $\sum_{i=1}^{n}f_{i}^{*}f_{i}=1$, and any $\mu_{1},\ldots, \mu_{n}\in\pr(G, C(X))$, we have $\sum_{i=1}^{n}f_{i}^{*}\mu_{i}f_{i}\in\pr(G, C(X))$.



\begin{remark}

Consider an inclusion $C(X) \subseteq A$ of unital $G$-$C^*$ algebras. In \cite{AU}, an action of a formal sum $\sum_{s\in G}\sum_{i\in I_s} f_i s f_i$ on $A$ is defined as follows:
\[
\mu \cdot a = \sum_{t\in G}\sum_{i\in I_s} f_i(s_i \cdot a)f_i.
\]
The authors claim that for a fixed $a\in A$, the set of all $\mu \cdot a$, where $\mu$ ranges over all formal sums, is $C(X)$-convex. However, $C(X)$-convexity is only retained as a property of the set and is not utilized throughout the paper.

For $\mu \in \text{Pr}(G,C(X))$, the corresponding formal sum is given by
\[
\sum_{t\in G} \sqrt{\mu(t)}~ t~ \sqrt{\mu(t)}.
\]
Hence, the corresponding action is defined as follows:
\[
\mu \cdot a = \sum_{t\in G} \sqrt{\mu(t)}(t \cdot a)\sqrt{\mu(t)}.
\]

We do not utilize this action throughout our paper. Instead, we employ the  following action:
\end{remark}
\begin{defn}\label{defn12}
Let $X$ be a compact Hausdorff space. Assume $C(X)\subseteq A$ is an inclusion of uintal $G$-$C^*$ algebras. For $\mu\in \pr(G, C(X))$ and $a\in $A, define
\[\mu.a=\sum_{t\in G}\mu(t) (t.a).\]
Clearly, the map $a\mapsto \mu.a$ is uintal an completely positive.
\end{defn}



It is important to note that for a $C^*$-algebra $A$, an $A$-valued measure can be defined on ( the Borel $\sigma$-algebra of ) any locally compact space. However, it is worth emphasizing that the definition and proof techniques for $A$-valued measures  on locally compact spaces differ significantly from those used for $A$-valued measures on discrete or compact spaces. As a result, we have chosen to discuss the notion of $A$-valued measures on locally compact spaces in a separate paper dedicated specifically to this topic.

\section{$\text{Ker}\left(G\curvearrowright\partial_F(G,X)\right)$ and Generalization of \cite[Theorem 5.1]{Ha}}
In the preliminaries, it was mentioned that $\text{Ker}(G\curvearrowright\partial_{F}G)=\text{Rad}(G)$, and the faithfulness of the action $G\curvearrowright\partial_{F}G$ is closely related to the existence of a unique trace on $C^{*}_{r}(G)$. In this section, we aim to extend and build upon these results.
\begin{prop}
\thlabel{amenradandker}
Let $X$ be a minimal $G$-space. Then,
\[\text{Ker}(G\curvearrowright\partial_F(G,X))=\text{Ker}(G\curvearrowright X)\cap \text{Rad}(G).\]
\begin{proof}

Let $\varphi:\partial_F(G,X)\to X$ denote the canonical quotient map. We can identify $C(X)$ as a subspace of $C(\partial_F(G,X))$ via $\varphi$. Consider an element $s \in \text{Ker}(G\curvearrowright\partial_F(G,X))$. This implies that $s\varphi(x)=\varphi(x)$ for all $x \in \partial_F(G,X)$. Since $\varphi$ is surjective, we can conclude that $s$ belongs to $\text{Ker}(G\curvearrowright X)$. Furthermore, it is known that for every  $x\in \partial_{F}(G, X)$, the stabilizer subgroup $G_x$ is amenable \cite[Proposition 3.3]{Kawabe}, and amenability is preserved under taking subgroups. Thus, we have that $\text{Ker}(G\curvearrowright\partial_F(G,X))\subseteq G_x$ is also amenable. Moreover, as $\text{Ker}(G\curvearrowright\partial_F(G,X))$ is a normal subgroup, we conclude that $\text{Ker}(G\curvearrowright\partial_F(G,X))\subseteq \text{Ker}(G\curvearrowright X)\cap \text{Rad}(G)$.

To show the other direction, let $x\in X$, and consider the set $\varphi^{-1}(x)$, which is $\text{Ker}(G\curvearrowright X)$-invariant. Let $\Lambda=\text{Ker}(G\curvearrowright X)\cap \text{Rad}(G)$. Since $\Lambda$ is an amenable subgroup of $\text{Ker}(G\curvearrowright X)$, it fixes a measure $\nu\in \text{Prob}(\varphi^{-1}(x))$. Moreover, as $\Lambda$ is a normal subgroup of $G$, we have $tg\nu=g\nu$ for all $t\in\Lambda$ and $g\in G$.

Consider $\nu$ as a measure on $\partial_F(G,X)$ and note that $\mathrm{supp}(\nu)\subseteq\varphi^{-1}(x)$. By utilizing \cite[Theorem 3.2]{Naghavi}, we have $t\delta_{y}=\delta_{y}$ for any $y\in\varphi^{-1}(x)$. So $t\delta_{y}=\delta_{y}$ for some $y\in\partial_{F}(G, X)$. Since $\partial_F(G,X)$ is minimal, we can conclude that $t\delta_{y}=\delta_{y}$ for all $y\in\partial_F(G,X)$. This completes the proof.
\end{proof}
\end{prop}
\begin{cor}\cite[Proposition 2.8]{BKKO}
$\text{Rad}(G)=\text{Ker}(G\curvearrowright\partial_FG)$.
\end{cor}
\begin{cor}
Suppose that $G\curvearrowright X$ is minimal. Assume that $\text{Rad}(G)=\{e\}$. Then, $\text{Ker}(G\curvearrowright\partial_F(G,X))=\{e\}$.
\end{cor}
\noindent
\begin{prop} \thlabel{2.1}
	Let $X$ be a minimal $G$-space and $t \not\in \text{Ker}(G\curvearrowright\partial_F(G,X))$. Then
	$$0\in \overline{\{\mu\lambda_{t} : \mu\in \pr_{f}(G, C(X))\}}.$$
\end{prop}
	\begin{proof}
		Suppose otherwise, which means there exists $\alpha > 0$ and $\omega \in (C(X)\rtimes_r G)^{*}$ such that $\mathrm{Re}\omega(\mu\lambda_{t}) \geq \alpha$ for all $\mu\in\pr_{f}(G, C(X))$. Since $t \notin \text{Ker}(G\curvearrowright\partial_F(G,X))$, there exists $y \in \partial_F(G,X)$ such that $ty \neq y$. By following the proof of \cite[Proposition 3.7]{AU} and \cite[Proposition 3.8]{AU}, we can find $\eta\in \overline{\{\omega\mu : \mu \in \pr_{f}(G, C(X))}\}^{\mathrm{w^*}}$ such that $\eta(\lambda_{t})=0$. Thus $0 \in \overline{\{\omega\mu : \mu \in \pr_{f}(G, C(X))}\}^{\mathrm{w^*}}$. However, this leads to a contradiction.

	\end{proof}

With the above proposition in hand, we are ready to give a generalization of \cite[Theorem 1.3]{BKKO}. Consider the following set which is a subset of $(C(X)\rtimes_{r} G)^*$.
$$S_P^{G}(C(X)\rtimes_r G)=\{\tau: \tau|_{C(X)}=\delta_x \text{ for }x\in X \text{ and }\tau|_{C_r^*(G)}\text{ is a trace}\}.$$

\begin{theorem}
	Let $X$ be a minimal $G$-space. The following are equivalent:
	\begin{itemize}
		\item[$(i)$] $\text{Ker}(G\curvearrowright\partial_F(G,X))=\{e\}$.
		\item[$(ii)$] for all $t\in G\setminus\{e\}$, $0\in \overline{\{\mu\lambda_{t} : \mu\in \pr_{f}(G, C(X))\}}$.
		\item[$(iii)$] every state $\tau\in S_P^{G}(C(X)\rtimes_rG)$ is of the form $\tau=\tau\circ\mathbb{E}$.
	\end{itemize}
	
	\begin{proof}
		$(i)\Rightarrow (ii)$ \thref{2.1}.\\
		
	    $(ii)\Rightarrow (iii)$ Let $\tau\in S_P^{G}(C(X)\rtimes_rG)$. Let $a \in C(X)\rtimes_rG$ and $\epsilon>0$. Then, there are $t_1,t_2,\ldots,t_n \in G\setminus\{e\}$ and $f_1,f_2,\ldots,f_n \in C(X)$ such that
	 	\[a\approx_{\epsilon}\sum_{i=1}^nf_i\lambda_{t_i}+\mathbb{E}(a)\]
		Since $0\in \overline{\{\mu\lambda_{t_{i}} : \mu\in \pr_{f}(G, C(X))\}}$ for each $i=1,2,\ldots,n$, we can find $s^{t_i}_1,s^{t_i}_2,\ldots,s^{t_i}_{m_i}\in G$ and $f^{t_i}_1,f^{t_i}_2,\ldots,f^{t_i}_{m_i}\in C(X)$ with $f^{t_i}_j\ge 0$, $\sum_{j=1}^{m_i}f^{t_i}_{j}=1$ such that
		\begin{equation}
		\label{avg}
		\left\|\sum_{j=1}^{m_i}f_j^{t_i}\lambda_{s^{t_i}_{j}}\lambda_{t_i}\lambda_{{s^{t_i}_{j}}^{-1}}\right\|<\frac{\epsilon}{\|f_i\|}.
		\end{equation}
		Now,
		\begin{align*}
		\tau\left(f_i\lambda_{t_i}\right)&=f_i(x)\tau(\lambda_{t_i})\\&=f_i(x)\tau\left(\sum_{j=1}^{m_i}f_j^{t_i}\lambda_{t_i}\right)&\text{$\left(\sum_{j=1}^{m_i}f^{t_i}_{j}=1\right)$}\\&=f_i(x)\tau\left(\sum_{j=1}^{m_i}f_j^{t_i}\lambda_{s^{t_i}_{j}}\lambda_{t_i}\lambda_{{s^{t_i}_{j}}^{-1}}\right)&\text{$(\tau|_{C_r^*(G)}$ is $G$-invariant)}\\
		\end{align*}
		Therefore, taking norm on both sides, we obtain that
		\begin{align*}
		\left|\tau\left(f_i\lambda_{t_i}\right)\right|&\le |f_i(x)|\left|\tau\left(\sum_{j=1}^{m_i}f_j^{t_i}\lambda_{s^{t_i}_{j}}\lambda_{t_i}\lambda_{{s^{t_i}_{j}}^{-1}}\right)\right|\\&\le |f_i(x)|\left\|\sum_{j=1}^{m_i}f_j^{t_i}\lambda_{s^{t_i}_{j}}\lambda_{t_i}\lambda_{{s^{t_i}_{j}}^{-1}}\right\|\\&\le |f_i(x)|\frac{\epsilon}{\|f_i\|}&\text{(Using equation \ref{avg})}\\&<\epsilon.
		\end{align*}
		Since $\epsilon>0$ is arbitrary, we see that $\tau(f_i\lambda_{t_i})=0$ for each $i=1,2,\ldots,n$. Hence,
		\[\tau(a)\approx_{\epsilon}\tau\left(\sum_{i=1}^nf_i\lambda_{t_i}+\mathbb{E}(a)\right)\approx_{\epsilon}\tau(\mathbb{E}(a))\]
		Again, since $\epsilon>0$ is arbitrary, we see that $\tau(a)=\tau(\mathbb{E}(a))$. \\
		
		$(iii)\Rightarrow (i)$ Suppose $\text{Ker}(G\curvearrowright\partial_F(G,X))$ is a non-trivial subgroup. By referring to \thref{amenradandker}, let's define $\Lambda$ as the intersection of the kernel of the action $G\curvearrowright X$ and the radical subgroup $\text{Rad}(G)$. We can conclude that $\Lambda$ is a non-trivial, amenable, and normal subgroup of $G$. Consequently, it fixes $\nu \in \text{Prob}(\varphi^{-1}(x))$, where $\varphi: \partial_F(G,X)\to X$ represents the corresponding factor map. It's worth noting that $\nu|_{C(X)}=\delta_x$. Additionally, we observe that $\Lambda$ is a normal, amenable subgroup of $G_x$. Let $\mathbb{E}_{\Lambda}$ be the canonical conditional expectation from $C_r^*(G_x)$ onto $C_r^*(\Lambda)$, $\tau_0$ be the unit character on $C_r^*(\Lambda)$, and $\mathbb{E}_x$ be the canonical conditional expectation from $C(X)\rtimes_rG$ onto $C_r^*(G_x)$. In this context, we find that $\tau_0\circ\mathbb{E}_{\Lambda}\circ\mathbb{E}_x=\tau$ is a state on $C(X)\rtimes_rG$ whose restriction to $C(X)$ is $\delta_x$. It's important to note that $\tau|_{C_r^*(G)}$ is $G$-invariant. For any $t\in G$, if $tst^{-1} \in \Lambda$ holds, then $s \in \Lambda$. Consequently, $\tau(\lambda_{tst^{-1}})=\tau(\lambda_s)=0$ or $\tau(\lambda_{tst^{-1}})=\tau(\lambda_s)=1$, depending on whether $s \in \Lambda$ or not. Thus, we can conclude that $\tau$ is $G$-invariant. It's worth mentioning that $\tau\ne\tau\circ\mathbb{E}$, as $\tau(\lambda_s)=1$ for $e\ne s\in\Lambda$.
	\end{proof}
\end{theorem}



We refer the readers to \cite{U} for more information about trace property of noncommutative crossed products.

\end{document}